\documentclass[a4paper, 11pt,reqno]{amsart}
\usepackage[top = 1in, bottom = 1in, left=1.3in, right=1.3in]{geometry}
\usepackage{setspace} 

\usepackage[utf8]{inputenc}
\usepackage[T1]{fontenc}

\usepackage{amssymb,amsmath,amsthm,graphicx,xcolor,mathtools,mathrsfs}

\usepackage{hyperref}
\hypersetup{colorlinks, linkcolor={red!50!black}, citecolor={green!50!black}, urlcolor={blue!50!black}}

\usepackage[
backend=biber,
natbib=true,
style=alphabetic,
isbn=false,
maxbibnames=9
]{biblatex}
\usepackage{csquotes}
\renewbibmacro{in:}{}
\DeclareFieldFormat{pages}{#1}

\DeclareFieldFormat{title}{\mkbibquote{#1}}
\DeclareFieldFormat[misc]{date}{Preprint (#1)}

\renewbibmacro*{doi+eprint+url}{
	\iftoggle{bbx:url}
	{\iffieldundef{doi}{\usebibmacro{url+urldate}}{}}{}
	\newunit\newblock
	\iftoggle{bbx:eprint}
	{\usebibmacro{eprint}}
	{}
	\newunit\newblock
	\iftoggle{bbx:doi}
	{\printfield{doi}}
	{}
}

\usepackage{cleveref}
\theoremstyle{plain}
\newtheorem{theorem}{Theorem}[section]
\crefname{theorem}{Theorem}{Theorems}

\crefname{proposition}{Proposition}{Propositions}

\newtheorem{corollary}[theorem]{Corollary}
\crefname{corollary}{Corollary}{Corollaries}

\newtheorem{lemma}[theorem]{Lemma}
\crefname{lemma}{Lemma}{Lemmas}

\crefname{conjecture}{Conjecture}{Conjectures}

\crefname{problem}{Problem}{Problem}

\crefname{claim}{Claim}{Claims}

\crefname{observation}{Observation}{Observations}

\crefname{setup}{Setup}{Setups}

\crefname{myth}{Myth}{Myths}

\crefname{fact}{Fact}{Facts}

\crefname{algorithm}{Algorithm}{Algorithms}

\newtheorem{remark}[theorem]{Remark}
\crefname{remark}{Remark}{Remarks}

\crefname{example}{Example}{Examples}

\theoremstyle{definition}
\newtheorem{definition}[theorem]{Definition}
\crefname{definition}{Definition}{Definitions}

\crefname{construction}{Construction}{Constructions}

\crefname{question}{Question}{Questions}

\numberwithin{equation}{section}

\usepackage[shortlabels]{enumitem}
\setlist[enumerate,1]{label={\upshape (\roman*)}}

\DeclareMathOperator{\spath}{sp}
\DeclareMathOperator{\spathC}{spc}

\usepackage{mathtools}

\newcommand{\Fam}{\mathcal{F}}
\newcommand{\N}{\mathbb{N}}
\newcommand{\Tannoying}{T^a}
\DeclareMathOperator{\prob}{P}
\DeclareMathOperator{\expectedvalue}{E}
\DeclareMathOperator{\variance}{Var}
\allowdisplaybreaks

\author[Arrepol et al.]{Francisco Arrepol}
\author[]{Patricio Asenjo}
\author[]{Ra\'ul Astete}
\author[]{V\'ictor Cartes}
\author[]{Anah\'i Gajardo}
\author[]{Valeria Henríquez}
\author[]{Catalina Opazo}
\author[]{Nicol\'as Sanhueza-Matamala}
\author[]{Christopher Thraves Caro}

\address[Arrepol, Asenjo, Astete, Cartes, Gajardo, Henríquez, Opazo, Sanhueza-Matamala, Thraves Caro]{Departamento de Ingeniería Matemática, Facultad de Ciencias Físicas y Matemáticas, Universidad de Concepción, Chile.}
\email{\{farrepol2016,
pasenjo2018,
rastete2018,
vcartes2017,
angajardo,
vhenriquez2016,copazo2018,nsanhuezam,cthraves\}@udec.cl}

\thanks{The research leading to these results was supported by ANID-Chile through the FONDECYT Iniciación Nº11220269 grant (N.~Sanhueza-Matamala).}

\title{Separating path systems in trees}

\begin{document}

\begin{abstract}
For a graph $G$, an edge-separating (resp. vertex-separating) path system of $G$ is a family of paths in $G$ such that for any pair of edges $e_1, e_2$ (resp. pair of vertices $v_1, v_2$) of $G$ there is at least one path in the family that contains one of $e_1$ and $e_2$ (resp. $v_1$ and $v_2$) but not the other.

We determine the size of a minimum edge-separating path system of an arbitrary tree $T$ as a function of its number of leaves and degree-two vertices. We obtain bounds for the size of a minimal vertex-separating path system for trees, which we show to be tight in many cases. We obtain similar results for a variation of the definition, where we require the path system to separate edges and vertices simultaneously. Finally, we investigate the size of a minimal vertex-separating path system in Erd\H{o}s--Rényi random graphs.
\end{abstract}

\maketitle

\section{Introduction}
Consider a communication network where one can send a monitoring packet to verify the availability of connections and nodes. If the packet arrives at its destination, all components are working, while if the packet fails, we know that one component is faulty, but we need to know which one. The problem, then, is to identify the faulty component. 

We consider the problem of identifying network components using paths. More precisely, we consider the following problem: Given a graph and a set of elements of the graph (edges and vertices), provide a set of paths so that one can completely identify each element using the set of paths to which it belongs. We say that such family \emph{separates} the target elements. It is natural to think that each monitoring packet has a cost.
Therefore, we are also looking for a set of paths that separates the network elements and has a minimum size. 

We consider two types of graphs: trees and random graphs. We use the Erd\H{o}s-R\'enyi model for generating random graphs $G(n,p)$. This model generates a random graph on $n$ vertices by including each of the $\binom{n}{2}$ possible edges independently with probability $p$ each.

We organize our presentation as follows. In Section \ref{sec:notation}, we present the definitions and notation used throughout the document. In Section \ref{sec:relatedwork}, we discuss the related work, present our contributions, and place them in context. In Section \ref{sec:construct}, we present constructions and structural results that we use later in the document. Section \ref{sec:Tedges} contains our results regarding the separation of edges in trees. Section \ref{sec:Tvertices} describes our results concerning separating vertices in trees and vertices and edges together in trees. Section \ref{sec:randomgraphs} presents our results regarding random graphs. Finally, in Section \ref{sec:conclusions}, we present our conclusions and a discussion of open problems. 

\section{Notation and definitions}\label{sec:notation}
We only consider finite, simple, and undirected graphs. The vertices and edges of a graph $G$ will be denoted by $V(G)$ and $E(G)$, respectively. We use $V$ and $E$ if the context clarifies the graph. We use $T=(V,E)$ to denote a tree, \textit{i.e.}, an acyclic connected graph, and $n$ and $m$ to denote the size of $V$ and the size of $E$, respectively. 

We say that $u$ and $v$ are \emph{neighbors}, or \emph{adjacent} vertices in $G$, when the edge $\{u,v\}$ is in $E(G)$. The number of neighbors of a vertex $v$ in a graph $G$ is its \emph{degree}, and we denote it as $d_G(v)$. We use the standard notation $uv$ to denote the edge $\{u,v\}$. A sequence of distinct vertices $v_0v_1v_2 \dotsb v_l$ is a \emph{path} from $v_0$ to $v_l$ provided that $v_{i-1}v_{i}\in E(G)$ for each $i \in \{1,2,\dotsc , l\}$. The \emph{length} of a path is its number of edges. If $P=v_0v_1\dotsb v_l$ is a path in $G$ we say that $P$ \emph{contains} the vertices $v_i$ for all $i \in \{0, \dotsc, l\}$ and the edges $v_{i-1}v_{i}$ for all $i \in \{1,2,\dotsc , l\}$. We also say that $P$ is a \emph{$v_0$-$v_l$-path} and that $v_0$ and $v_l$ are the \emph{extremes} of $P$. Given a vertex $v$, we abbreviate `$P$ contains $v$' by writing $v \in P$, similarly for an edge $e$, `$e \in P$' means that $P$ contains $e$.

\begin{definition}
Let $G=(V,E)$ be a graph and $S\subseteq V\cup E$ be a set of elements of $G$. Let $\mathcal{F}$ be a family of paths in $G$. We say that $\mathcal{F}$ \emph{separates} $S$ if for every pair of elements $s$ and $t$ in $S$ there is a path $P$ in $\mathcal{F}$ such that 
\[
(s\in P \mbox{ and } t \not\in P)\, \mbox{ or } \, (s \not\in P \mbox{ and } t \in P).
\]
We say that a family $\mathcal{F}$ that separates $S$ also \emph{covers} $S$ if for all $s\in S$, there is a path
$P\in \mathcal{F}$ that contains $s$.
\end{definition}

A useful notation is the following. For $s\in S$, let $\mathcal{F}(s)\subseteq \mathcal{F}$ denote the set of paths in $\mathcal{F}$ that contain $s$.
Therefore, $\Fam$ covers $S$ if and only if $\mathcal{F}(s)\not = \emptyset$ for all $s \in S$, and $\Fam$ separates $S$ if and only if $\Fam(s)\not=\Fam(t)$ for all different $s$ and $t$ in $S$.

For brevity, if $\mathcal{F}$ is a family of paths that separates $S$, we say that $\mathcal{F}$ is an \emph{$S$-separating system}, and if it separates and covers $S$ we say that $\Fam$ is an \emph{$S$-separating-covering system}.
We will also say \emph{vertex-separating system} or \emph{edge-separating system} if $S$ is the set of vertices or edges of a graph, respectively.

Given a graph $G$ and a set of elements $S$ of $G$, we want to find $S$-separating systems and $S$-separating-covering systems of minimal size.
We define
\begin{eqnarray*}
\spath(G,S)&:=& \min\{|\mathcal{F}|\colon \mathcal{F} \mbox{ separates } S\}, \text{ and}\\
\spathC(G,S)&:=& \min\{|\mathcal{F}|\colon \mathcal{F} \mbox{ separates and covers } S\}.
\end{eqnarray*}
Thus, given a graph $G = (V,E)$, we naturally write $\spathC(G,V)$ and $\spathC(G,E)$ to refer to the minimum size of a $V$-separating-covering system or $E$-separating-covering system, respectively.
We use this notation even if $V$ or $E$ were not specified before.

Given a tree $T$, we use $h_1(T)$ and $h_2(T)$ to denote the number of leaves and the number of
vertices with degree $2$ in $T$, respectively.
We use simply $h_1$ and $h_2$ when $T$ is contextually clear.
We say that an edge $e=uv \in E(T)$ is an \emph{interior} edge if neither $u$ nor $v$ are leaves, and denote the set of all interior edges of a tree $T$ by $E^*(T)$.
A \emph{bunch of leaves}, or simply a \emph{bunch}, in a tree $T$ is a non-trivial connected component of the subgraph obtained from $T$ after deleting all interior edges (a non-trivial connected component of a graph is a connected component that is not an isolated vertex). 
The \emph{size of a bunch of leaves} is the number of leaves (in $T$) it contains.

A \emph{bare path} in a tree is a path such that only its two extreme vertices have a degree different from two.
Note that each tree $T$ has a unique set of bare paths that partitions $E(T)$.
Let $\mathcal{P} = \{P_1,\ldots, P_r\}$ be the family of bare paths in $T$ that partitions $E(T)$.
We use $\mathcal{I} \subseteq \mathcal{P}$ to denote the set of bare paths which do not include a leaf and have at least two edges.
We define the parameter
  \[
  h^\ast_2(T)=\sum_{P_i\in \mathcal{P} \setminus \mathcal{I}} \left(|E(P_i)|-1\right)   + \sum_{P_i\in \mathcal{I}} \left(|E(P_i)|-2\right).
  \]
Note that $0\leq h^\ast_2(T) \leq h_2(T)$.
Essentially, $h^\ast_2(T)$ counts every degree-$2$ vertex from each bare path containing a leaf; and counts all but one vertex of degree two from each bare path in $\mathcal{I}$.
From this, it holds that $h^\ast_2(T) = h_2(T) - | \mathcal{I}|$.

\section{Related work and our results}\label{sec:relatedwork}
\subsection{Related work}

The notion of separating families of paths originates from the older and more general concept of \emph{separating families of sets} introduced by Rényi~\cite{Renyi1961}.
Given a \emph{ground set} $S$, a \emph{separating family} is a collection of sets $A_1, \dotsc, A_m \subseteq S$ such that for every two distinct $i,j \in S$, there is a set $A_k$ such that either $i \in A_k$ and $j \notin A_k$, or $i \notin A_k$ and $j \in A_k$.

The problem becomes more intricate if we impose restrictions on the sets which are allowed.
A natural way of imposing restrictions is to take a graph. Consider the set of vertices or edges of the graph as a ground set, and look for separating families 
where the subsets must satisfy certain graph-theoretic restrictions.

Problems of this sort were considered in the theoretical computer science literature under multiple names, including ``identifying families'', ``test covers'', ``test families'', and ``test sets''; see~\cite {BHHHLRS2003} for an overview.
As explained before, the motivation was the detection of faulty processors in a network.
As far as we are aware, the use of separating families of paths was suggested first by Zakrevski and Karpovsky~\cite{ZakrevskiKarpovsky1998}.
Later, Honkala, Karpovsky, and Litsyn~\cite{HonkalaKarpovskyLitsyn2003}, independently Rosendahl~\cite{Rosendahl2003} considered both vertex and edge separation using cycles and investigated this problem in specific families of graphs, as hypercubes, complete bipartite graphs, and grids.

Foucaud and Kov\v{s}e~\cite{FoucaudKovse2013} considered the problem of finding families of paths which separate and cover the vertices of a graph (under the name ``identifying path covers'').
They considered algorithmic and combinatorial points of view.
On the algorithmic side, they showed that if the families of paths to be considered are also required to have bounded size, the associated optimization problem is APX-complete.
Using our notation, they obtained results for $\spathC(G, V)$ in various cases.

\begin{theorem}[\cite{FoucaudKovse2013}]
The following hold:
    \begin{enumerate}
        \item For an $n$-vertex clique $K$, $\spathC(K, V) = \lceil \log_2(n+1) \rceil$.
        \item For an $n$-vertex path $P$, $\spathC(P,V) = \lceil (n+1)/2 \rceil$.
        \item For an $n$-vertex cycle $C$, $\spathC(C,V) = n - 1$ if $3 \leq n \leq 4$, and $\spathC(C,V) = \lceil n/2 \rceil$ if $n \geq 5$.
        \item For any $n$-vertex connected graph $G$, $\spathC(G,V) \leq \lceil 2n/3 \rceil + 5$. 
    \end{enumerate}
\end{theorem}

Later, the problem also attracted the interest of the combinatorics community
when Katona (see \cite{FKKLN2014}) raised the problem of estimating the parameter $\spath(n)$, which is the maximum of $\spath(G, E)$ taken over all $n$-vertex graphs $G$.
Falgas-Ravry, Kittipassorn, Korándi, Letzter, and Narayanan~\cite{FKKLN2014} conjectured that $\spath(n) = O(n)$ and proved that $\spath(n) = O(n \log n)$ holds.
They also obtained results for graphs of large minimum degree, random graphs, certain families of quasi-random graphs, and trees.
We state their results for random graphs and trees which are relevant to our work.
Given a function $p: \N \rightarrow [0,1]$, a sequence of random graphs $G_n = G(n,p(n))$ and a sequence of events $E = \{ E_n\}_{n \in \N}$, we will say that ``$G(n,p(n))$ satisfies $E$ \emph{with high probability}'' if $\prob[G_n \in E_n] \rightarrow 1$ as $n \rightarrow \infty$.
Sometimes we abbreviate this to w.h.p.

\begin{theorem}[\cite{FKKLN2014}]
    The following hold:
    \begin{enumerate}
        \item For any $p = p(n)$, w.h.p., the random graph $G = G(n,p)$ satisfies $\spath(G,E) \leq 48n$.
        \item For every tree $T$ on $n \geq 4$ vertices, 
        $\left\lceil \frac{n+1}{3} \right\rceil \leq \spath(T,E) \leq \left\lfloor \frac{2(n-1)}{3} \right\rfloor$, and both bounds are attained with equality by some $n$-vertex tree.
    \end{enumerate}
\end{theorem}

Independently, Balogh, Csaba, Martin, and Pluhár~\cite{BCMP2016} introduced a different definition of ``strong'' edge-separating systems (which in particular are edge-separating).
They also conjectured that $O(n)$ paths should be enough to form a strong separating family of the edges of any $n$-vertex graph.
They investigated trees, cliques, random graphs, and hypercubes.
Wickes~\cite{Wickes2022} investigated edge-separating path systems in cliques, in particular proving that $\spath(K_n, E) \leq n$ for infinitely many values of $n$.
Letzter~\cite{Letzter2022} proved that $\spath(n) = O(n \log^\ast n)$, where $\log^\ast n$ is the iterated-log function.
Finally, Bonamy, Botler, Dross, Naia, and Skokan~\cite{BBDNS2023} showed that $\spath(n) \leq 19n$.
Their proof, in fact, gives strong separating path systems and thus resolves both the conjecture of Falgas-Ravry et al. and that of Balogh et al.

\subsection{Our results}
We investigate the parameters $\spathC(T,E)$, $\spathC(T,V)$ in more detail when $T$ is a tree.
Our first main result, proved in Section~\ref{sec:Tedges}, establishes the value of $\spathC(T,E)$ precisely in every case.
For most cases, we can express the result in terms of the number of leaves and vertices of degree $2$ of a given tree, but interestingly, the formula fails only in a single case.
Let $\Tannoying$ be the \emph{binary tree of depth $2$}, that is, the tree on the vertex set $\{1, 2, 3, 4, 5, 6, 7\}$ with edge set $\{ 12, 13, 24, 25, 36, 37 \}$.
Note that $h_1(\Tannoying) = 4$ and $h_2(\Tannoying) = 1$.

\begin{theorem} \label{theorem:edgesepcover-trees}
    We have $\spathC(\Tannoying, E) = 4$, and for every tree $T \neq \Tannoying$,
    \[ \spathC(T,E) = \max \left\lbrace \left\lceil \frac{2 h_1(T) + h_2(T)}{3} \right\rceil, \left\lceil \frac{h_1(T) + h_2(T)}{2} \right\rceil \right\rbrace. \]
\end{theorem}

In Section~\ref{sec:Tvertices}, we obtain lower and upper bounds on $\spathC(T,V)$ in terms of $h_2^\ast(T)$.
Let $K_{1,3}$ be the tree corresponding to a \emph{star with three leaves}.

\begin{theorem}\label{theorem:edgevertexsep-trees}
Let $T$ be a tree.
Then, 

   \[ \spathC(T,V)  \geq 
  \max\left\{ \left\lceil\frac{2h_1(T) + h_2^\ast(T)}{3}\right\rceil , \left\lceil\frac{h_1(T) + h_2^\ast(T)}{2}\right\rceil \right\}.\]

  Additionally, if all the bunches of $T$ have size at least three and $T \neq K_{1,3}$, then,

  \[  \spathC(T,V) 
        \leq 
               \left\lceil \frac{2h_1(T)}{3}\right\rceil + \left\lceil \frac{h_2^\ast(T)+1}{2}\right\rceil.\]
       
\end{theorem}
The restrictions in the second part of \Cref{theorem:edgevertexsep-trees} cannot be avoided. In \Cref{corollary:c13star}, we describe a family of trees which includes $K_{1,3}$, all the trees in the family (apart from $K_{1,3}$) have only bunches of size $2$, and no tree in the family satisfies the upper bound of \Cref{theorem:edgevertexsep-trees}.

Finally, in Section~\ref{sec:randomgraphs}, we investigate the values of $\spath(G,V)$ and $\spathC(G, V)$ when $G = G(n,p)$ is a random binomial graph.
These values satisfy a certain ``threshold'' behavior which is presented in the following result.

\begin{theorem} \label{theorem:random}
    Let $G = G(n,p)$ be a binomial random graph.
    \begin{enumerate}
        \item \label{item:randomupper} If $p \geq (2\ln n + \omega(\ln \ln n))/n$, then w.h.p. $\spathC(G,V) \leq \lceil \log_2 n \rceil + 1$.
        \item \label{item:randomlower} If $p \leq (\ln n - \omega(\ln \ln n))/n$, then w.h.p. $\spath(G,V) = \omega(\log n)$.
    \end{enumerate}
\end{theorem}

\section{Constructions}\label{sec:construct}
In this section, we present the constructions that will be used to prove the results in sections \ref{sec:Tedges} and  \ref{sec:Tvertices}, and we establish their properties.

\begin{definition}[ABC construction]\label{cyclic-ABC}
Let $T$ be a tree such that $|V(T)|\ge 3$ and $h_1(T)= 3k$. We use a depth-first search (DFS) to assign names to the leaves of T as follows: Select any leaf as the starting point of the DFS and label it as $a_1$. Then, proceed with the DFS and name the leaves in order of appearance as $a_2, a_3 \dots, a_k, b_1, \dots, b_k, c_1, \dots, c_k$. We define the family of paths $\mathcal{F}=\{P_1,...,P_k,Q_1,...,Q_k\}$, where $P_i$ is the only path between $a_i$ and $b_i$, and $Q_i$ is the only path between $a_i$ and $c_i$.
\end{definition}

\begin{remark}\label{rem:consecutiveABC}
Since the DFS visits each edge twice, we can observe that for any edge $e \in E(T)$, the leaves of the two connected components of $T \setminus  e$ are those that are named between the two visits, as well as the leaves that are named before and after the two visits. Consequently, the leaves of these two connected components will be consecutively named in the order provided by the DFS traversal. 
\end{remark}

\begin{lemma}\label{lema 4.3}
Let $T$ be a tree with $h_2(T)=0$ and $h_1(T)=3k$.
Then the family $\mathcal{F}$ of paths obtained by the ABC construction is an edge-separating-covering system and $|\mathcal{F}|= 2k$.
\end{lemma}

\begin{proof}
Let $\Fam$ be the family of paths obtained from the ABC construction in Definition \ref{cyclic-ABC} applied to $T$.
The size of $\mathcal{F}$ is $2k$, and we will show that $\Fam$ is an edge-separating-covering system.

Suppose $a_1, \dots, a_k, b_1, \dots, b_k, c_1, \dots, c_k$ is the enumeration of the leaves of $T$ given by the ABC construction.
Let $e,f \in E(T)$ be two distinct edges in $T$, then $T\setminus\{e,f\}$ is a forest with three connected components: $T_1$ that contains one endpoint of $e$, $T_2$ that contains one endpoint of $e$ and one endpoint of $f$, and $T_3$ that contains one endpoint of $f$. 

Since $T$ does not have vertices of degree two, all leaves of $T_1, T_2$ and $T_3$ are also leaves of $T$. 
Using Remark \ref{rem:consecutiveABC} two times, we can conclude that the set of leaves of $T_1$ and $T_3$ are consecutive in the cyclic order of the leaves of $T$.
We express this by saying that $T_1$ and $T_3$ have an \emph{interval of leaves} (which is not necessarily true for $T_2$, but we can assert that the leaves of $T_2$ are the union of at most two intervals of leaves).
We will prove that there is at least one path in $\mathcal{F}$ that starts in $T_j$ and ends outside $T_j$, for each $j\in\{1,2,3\}$, proving in this way that $\mathcal{F}$ both separates and covers $e$ and $f$.

In order to prove this, let us fix $j\in\{1,2,3\}$ and suppose that all paths with an endpoint in $T_j$ have both endpoints inside $T_j$.
This implies that each time one of $a_i$, $b_i$, or $c_i$ is a leaf in $T_j$, then the three of $a_i$, $b_i$ and $c_i$ are leaves in $T_j$.
If $a_i\in T_j$, we distinguish three cases: either $a_i,a_{i+1},\dots,a_k,b_1,\dots,b_i$ are all leaves of $T_j$, or $b_i,b_{i+1},\dots,b_k,c_1,\dots,c_i$ are all contained in $T_j$, or $c_i,c_{i+1},\dots,c_k,a_1,\dots,a_i$ are all in $T_j$.
In any case, for every $l\in\{1,\dots,k\}$, either $a_l$, $b_l$ or $c_l$ is in $T_j$, implying that $T_j=T$, which is a contradiction.

This proves that the family given by the ABC construction separates and covers the edges of $T$, as required.
\end{proof}

Next, we introduce a construction fully based on a planar drawing of a tree $T$
(a similar construction appears in \cite[Section 2.3]{BCMP2016}).

\begin{definition}[Planar construction]\label{Construction-planar}
    Given a tree $T$, take any planar drawing of $T$ that presents the leaves in a cyclic order.
    Starting from any leaf, enumerate the leaves of $T$ as $v_1, \dots, v_{h_1}$, clockwise.
    Then, for each $1\leq i \leq h_{1}-1$, we define the path $P_i$ that starts in $v_i$ and ends in $v_{i+1}$, and also the path $P_{h_1}$ that starts in $v_{h_1}$ and ends in $v_1$.
    This defines the \emph{Planar construction} of a family of $h_1$ paths in $T$.
\end{definition}

The Planar construction can be defined in an equivalent way as follows.
Starting from the planar drawing of $T$, add a new vertex $\infty$ located far from $T$ and connect it to each leaf of $T$ to form a new graph $G$.
This gives a planar drawing of $G$.
Since $G$ is $2$-connected, each face is bounded by a cycle;
so for any face $F$ of $G$, we can consider the cycle $C$ containing all vertices from $F$.
Such a cycle contains exactly two leaves of $T$, and (by removing $\infty$) it corresponds naturally to a unique path in the system $\Fam$ from Definition~\ref{Construction-planar}.

\begin{remark}\label{grado-cara-planar}
We can see that, in $G$, the number of faces that contains a vertex $v$ is equal to the degree of $v$, and that each edge belongs to exactly two faces.
Consequently, any vertex in $T$ is contained in as many paths in $\Fam$ as its degree, except for the leaves, which belong exactly to two different paths.
Similarly, any edge is contained in exactly two different paths of $\Fam$.
\end{remark}

\begin{lemma} \label{lema-planar-general}
    Let $T$ be a tree with $h_2(T)=0$ and $h_1(T)\ge 3$.
    The family $\Fam$ of paths obtained by the Planar construction has size $h_1(T)$ and separates and covers $E(T)$, and $V(T) \cup E^*(T)$. 
\end{lemma}
\begin{proof}

From \Cref{grado-cara-planar}, it is clear that $\Fam$ covers $V\cup E$. The size of $\Fam$ is $h_1(T)$.
The proof that $\Fam$ separates $E$ is given in \cite[Proof of Theorem 5]{BCMP2016}.

Let $u$ and $w$ be two different vertices in $T$.
If they are adjacent, at most one of them is a leaf, say $u$, and then $d_T(w)\geq 3$. Therefore, $w$ belongs to a path of $\Fam$ that does not contain $u$.
If $u$ and $w$ are not adjacent, there is $k \geq 1$ and a unique path $P=uu_1\dotsb u_kw$ from $u$ to $w$.
We know that there is a path $P^*\in\Fam$ that separates $uu_1$ from $u_kw$.
Let us suppose without loss of generality that $P^*$ contains $uu_1$ but does not contain $u_kw$.
Such path separates $u$ from $w$, indeed, $P^*$ contains $u$, but since the unique path from $u$ to $w$ contains also $u_kw$, $P^*$ cannot contain $w$.
To conclude this proof, take $e \in E^*$ and $v\in V$.
We distinguish two cases. \medskip
        
\noindent \emph{Case 1: $v$ is a leaf.} 
Let $e_v$ be the only edge adjacent to $v$.
    This edge is not internal; thus, it is different from $e$.
    Therefore, there exists a path $P^*$ in $\Fam$ that separates $e$ from $e_v$.
    If $P^*$ contains $e$ and not $e_v$, it cannot contain $v$.
    If $P^*$ contains $e_v$ and not $e$, it contains $v$ and not $e$.
    In any case $P^*$ separates $e$ from $v$.
    \medskip
    
\noindent \emph{Case 2: $v$ is not a leaf.}
    Since $v$ is not a leaf, $d_T(v)\geq 3$. This means $v$ is in at least three paths in $\Fam$.
    As $e$ is an edge, $e$ is in only two paths of $\Fam$, meaning there is a path $P^* \in \Fam$ such that $e \notin P^*$ and $v \in P^*$.
    This proves that $\Fam$ separates $E$, and $V\cup E^*$.
\end{proof}

The Planar construction uses one path for each leaf, while the ABC construction is more efficient because it uses only two paths for every three leaves. The heart of the ABC construction is to join three leaves with two paths. Such a construction automatically covers and separates the three leaves. Given three leaves $u, v, w$ of a tree $T$, we define a \emph{$uvw$-seagull} to be the set consisting of the unique $u$-$v$-path in $T$ and the unique $v$-$w$-path in $T$.

\begin{definition}[Bunch construction]
    \label{caso-planar-2}
Let $T$ be a tree different from the bi-clique $K_{1,3}$, such that all its bunches have a size of at least two. Enumerate the bunches of $T$ from $0$ to $r-1$. 
Enumerate the leaves of bunch $i$ from $v^i_1$ to $v^i_{\smash{\ell_i}}$.
Proceed in two steps for this construction.
In the first step, for each $i$, include the $v^i_{\smash{\ell_i - 1}}$-$v^i_{\smash{\ell_i}}$-path, and the $v^i_{\smash{\ell_i}}$-$v^{i+1}_1$-path (where sums are $\mod r$).
In the second step, group all remaining leaves into groups of size three and a group of size smaller than three, if necessary. 
Inside each group of size three, include the two paths of a seagull corresponding to those leaves.
For the group of size less than three, cover each leaf and its unique neighbor with the single edge touching the leaf.
In the case $r=1$, proceed directly with the second step. 
\end{definition}

\begin{lemma}
\label{lema-planar-racimos}    
   Consider a tree $T$ distinct from $K_{1,3}$ that satisfies the conditions $h_2(T) = 0$, and all of its bunches have a size of at least three. Let $\Fam$ be the family of paths obtained through the Bunch construction. Then, $\Fam$ covers and separates both $E(T)$ and $V(T) \cup E^*(T)$. Furthermore, 
$|\Fam| = \left\lceil \frac{2h_1(T)}{3} \right\rceil$.
\end{lemma}
\begin{proof}
First, we study the size of $\Fam$. Let us say that $h_1(T) = 3k + i$, with $i \in \{0,1,2\}$. From the construction, we have $|\Fam| = 2k + i = \left\lceil\frac{2h_1(T)}{3}\right\rceil$.

Now, we need to show that $\Fam$ covers and separates $E(T)$ and $V(T) \cup E^\ast(T)$. Let $T'$ be the tree obtained from $T$ by deleting leaves until all bunches have size two. We can do this because we assume that all the bunches of $T$ have a size of at least three. Apply the Planar construction to $T'$ to obtain a family $\Fam'$. Thus, $\Fam'$ covers and separates $E(T')$ and $V(T') \cup E^\ast(T')$.

If an element $s$ (a vertex or an edge) does not belong to a bunch, then $\Fam'(s) \subseteq \Fam(s)$\footnote{Even though the sets are defined over different trees, we consider the obvious mapping between paths in $T'$ and paths in $T$.}. Therefore, all those elements are covered and separated by $\Fam$, since $\Fam'$ covers and separates $V(T') \cup E^\ast(T')$. The Bunch construction covers all elements that belong to a bunch. Therefore, $\Fam$ covers $V(T) \cup E(T)$.

Now, consider two elements $s$ and $b$ in $V(T) \cup E(T)$. By the previous observation, if neither $s$ nor $b$ belongs to a bunch, they are separated by $\Fam$. Assume that $b$ belongs to a bunch and $s$ does not. We use $\Fam_1$ to denote the family obtained at the end of the first step of the Bunch construction. Then, we have $\Fam_1(s) = \Fam'(s)$, and by \Cref{grado-cara-planar}, $|\Fam'(s)| \geq 2$. If $|\Fam_1(b)| \leq 1$, then there is a path in $\Fam_1(s)$ that does not contain $b$, and therefore, $b$ and $s$ are separated by $\Fam_1$. If $|\Fam_1(b)| \geq 2$, then $b$ is either the last leaf of the bunch, its adjacent edge, or its neighbor in the bunch. In all these cases, $\Fam_1(b)$ contains the path from the last leaf of the bunch to the leaf before in the same bunch. That path stays in the bunch and therefore does not contain $s$. Hence, $\Fam_1$ separates $b$ and $s$. To conclude this part of the proof, we notice that if a family separates two elements and we add a path to that family, the new family still separates the two elements. Therefore, $\Fam$ separates $b$ and $s$.

Finally, let $s$ and $b$ be two elements of a bunch. If $s$ is a vertex but not a leaf, $\Fam(s)$ contains the union of $\Fam(v)$, where $v$ ranges over all leaves that belong to the bunch of $s$, and when $T$ is different from $K_{1,3}$, at least two of these sets are disjoint. When all bunches have a size of at least three, the Bunch construction is the union of seagulls and possibly up to two more one-edge paths covering some leaves (if the number of leaves is not divisible by three). In any case, $\Fam$ separates all leaves. Using this fact and the previous observation, we conclude that $\Fam$ separates $s$ and $b$ unless $s$ is a leaf and $b$ is the edge adjacent to $s$, or vice versa. In conclusion, $\Fam$ covers and separates $E(T)$ and $V(T) \cup E^\ast(T)$.
\end{proof}

\section{Edge-separating path systems in trees}\label{sec:Tedges}

In this section, we will prove \Cref{theorem:edgesepcover-trees}. Recall that $\Tannoying$ is the binary tree of depth $2$, and we defined it before the statement of \Cref{theorem:edgesepcover-trees}. This tree is the only tree where the formula of \Cref{theorem:edgesepcover-trees} fails and thus requires a separate analysis.

\begin{lemma}
    $\spathC(\Tannoying, E) = 4$.
\end{lemma}

\begin{proof}
    Suppose that $\Fam = \{P_1, P_2, P_3\}$ is an edge-separating-covering system.
    The vertex $1$ has degree $2$; hence, it must be the endpoint of some path (if not, the two edges adjacent to such a vertex would not be separated), say $P_1$, in $\Fam$.
    By symmetry, we can assume $P_1$ contains the edge $12$ and does not contain the edge $25$.
    Every edge containing a leaf must be covered by some path; thus, the leaves $5$, $6$, and $7$ must be the endpoint of the paths $P_2$ and $P_3$.
    Since $\Fam$ is a separating family, it must hold that one of $\{5, 6, 7\}$ is an endpoint of both $P_2$ and $P_3$; for the remaining two leaves one is an endpoint of $P_2$, and the other one is an endpoint of $P_3$.
    In any case, $\Fam$ does not separate the edges $25$ and $13$, leading to a contradiction.
    Therefore, $\spathC(\Tannoying, E) \geq 4$.
    On the other hand, the family $\{ 124, 125, 136, 137\}$ is an edge-separating-covering system, thus $\spathC(\Tannoying, E) \leq 4$.
\end{proof}

\subsection{Lower bounds}
We begin
by showing a lower bound on $\spathC(T,E)$ which holds for any tree $T$.

\begin{lemma} \label{lem:cotainf}
Let $T$ be a tree with $h_1$ leaves and  $h_2$ vertices of degree two.
Then
$\spathC(T,E)\geq\max\left\{(2h_1+h_2)/3,(h_1+h_2)/2\right\}$.
\end{lemma}

\begin{proof}
Let $\mathcal{F}$ be an $E(T)$-separating-covering system.
The bound
$ \frac{2h_1+h_2}{3}\leq |\mathcal{F}|$
appears inside the proof of Theorem 1.6 in \cite[page 349]{FKKLN2014}.
On the other hand, every leaf of $T$ must be the endpoint of some path in $\mathcal{F}$ (otherwise, there would be an uncovered leaf
), and also every vertex of degree two of $T$ must be the endpoint of some path in $\mathcal{F}$, as noted before.
As every path has two endpoints, we immediately have that $h_1 + h_2 \leq 2 |\mathcal{F}|$, as required.
\end{proof}

It remains to show that the upper bound of \Cref{theorem:edgesepcover-trees} holds. We will proceed by considering different cases.
\subsection{More leaves than vertices of degree 2}
We first consider trees with more leaves than vertices of degree two.
The structure of our proof makes it natural to consider the following sub-case first.

\begin{lemma} \label{lema-arista-h1igualh2}
Let $T$ be a tree such that $h_1(T) \geq h_2(T)$, and every leaf is adjacent to a vertex of degree two. Then,
$\spathC(T,E) \leq \left\lceil (h_1(T)+h_2(T))/2\right\rceil = \left\lceil (2h_1(T)+h_2(T))/3\right\rceil$.
\end{lemma}

\begin{proof}
    Let $h_1=h_1(T)$.
    Consider any planar drawing of $T$ with its leaves in a circumference, and enumerate the leaves of $T$ as $v_1, \dotsc, v_{h_1}$ in clockwise order according to the drawing.
    For each $1 \leq i \leq h_1$, let $u_i$ be the unique vertex adjacent to $v_i$, which by assumption must be a vertex of degree $2$ in $T$.
    
    If there are $1 \leq i < j \leq h_1$ such that $u_i = u_j$, then $T$ actually consists only of three vertices $v_i u_i v_j$ forming a two-edge path, thus $h_1(T) = 2$ and $h_2(T) = 1$.
    In this case, $\{ v_i u_i, u_i v_j\}$ is an $E(T)$-separating-covering system of size $2 = \lceil 3/2 \rceil = \lceil (h_1(T)+h_2(T))/2 \rceil = \lceil (2h_1(T)+h_2(T))/3 \rceil$, as required.
    Thus, we can assume that all the vertices $u_i$ with $1 \leq i \leq h_1$ are distinct.
    Since $h_1(T) \geq h_2(T)$, this implies that $h_2(T) = h_1(T)$.
    In particular, $\lceil (h_1(T)+h_2(T))/2 \rceil = \lceil (2h_1(T)+h_2(T))/3 \rceil = h_1$, thus it is enough to find an $E(T)$-separating-covering system of size $h_1$.
    
    Consider the family $\Fam = \{ P_1, \dotsc, P_{h_1}\}$ where for each $1 \leq i \leq h_1$, the path $P_i$ goes from $v_i$ to $u_{i+1}$ (where we set $u_{h_1 + 1} = u_1$).
    Then $\Fam$ has size $h_1(T)$, we now show that it is an $E(T)$-separating-covering system.
    
    If $h_1 = 2$, then (since every tree has at least $\Delta(T)$ leaves) it must happen that $\Delta(T) \leq 2$, thus $T$ is a path.
    Since $h_1 = h_2 = 2$, $T$ is actually a three-edge path $v_1 u_1 u_2 v_2$.
    It is straightforward to check that $\Fam$ covers and separates $E(T)$ in this case.
    Thus, from now on, we assume $h_1 \geq 3$.
    
    Now, obtain a tree $T'$ with a path system $\Fam'$ by removing all leaves $\{v_1, \dotsc, v_{h_1}\}$ and edges $\{ v_i u_i : 1 \leq i \leq h_1 \}$ from $T$ and from every path in $\Fam$.
    Since $h_1(T) = h_2(T)$, we know that there are no more degree-$2$ vertices in $T$ apart from $\{u_1, \dotsc, u_{h_1}\}$; and by removing the leaves every vertex $u_i$ will become a leaf in $T'$.
    Thus $h_1(T') = h_1$ and $h_2(T') = 0$.
    Every path in $\Fam'$ now goes from $u_i$ to $u_{i+1}$, thus in fact 
    we have that $\Fam'$ corresponds precisely to the Planar construction (Definition~\ref{Construction-planar}).
    Thus, since $h_1(T') = h_1 \geq 3$, Lemma~\ref{lema-planar-general} implies that $\Fam'$ is an $E(T')$-separating-covering system.
    Remark~\ref{grado-cara-planar} yields that each edge of $T'$ belongs to exactly two paths of $\Fam'$.
    Also, in $\Fam$ we have that the edge of $T$ adjacent to the leaf $v_i$ belongs only to the path $P_i$.
    Putting all together, this implies that all edges of $T$ are covered and separated by $\Fam$.
\end{proof}

Given a leaf $u$ in a tree $T$, we say that $u$ is \emph{useful} if its unique neighbor in $T$ does not have degree $2$.
Since the case where $h_1(T) \geq h_2(T)$ and no useful leaves is covered by \Cref{lema-arista-h1igualh2}, we will now consider cases where there is at least one useful leaf.

The first case to consider is the case where there are no vertices of degree $2$.

\begin{lemma}\label{lema-arista-h1mayorh2}
Let $T$ be a tree such that $h_2(T) = 0$.
Then, \[\spathC (T,E) \leq \left\lceil 2h_1(T)/3 \right\rceil.\] 
\end{lemma}
\begin{proof}
    If $h_1(T)$ is divisible by $3$, then the ABC construction and \Cref{lema 4.3} shows that $\spathC(T,E) \leq 2 h_1(T)/3$.
Thus, we can assume that $h_1(T)$ is not divisible by $3$.
Write $h_1(T) = 3q+s$ with $q \geq 0$ and $s \in \{1,2\}$.

Suppose first that $s = 2$.
If every vertex of $T$ is a leaf, then $T$ must be a single edge and $\spathC(T,E) \leq 1$.
Otherwise, we can build $T'$ obtained by adding a new leaf $u$, joined to an arbitrary non-leaf of $T$.
Then $h_2(T') = h_2(T) = 0$ and $h_1(T') = h_1(T)+1$ is divisible by $3$, so there exists an edge-separating family $\mathcal{F}'$ for $T'$ of size at most $2 h_1(T')/3 = \lceil 2 h_1(T)/3 \rceil$.
This can be transformed into an edge-separating family $\mathcal{F}$ of $T$ by removing $u$ and its adjacent edge from every path of $\mathcal{F}'$ which contains it.
Thus $\spathC(T,E) \leq \lceil 2 h_1(T)/3 \rceil$ in this case.

Now suppose that $s = 1$.
Take any leaf $u$ of $T$, and let $w$ be its unique neighbor.
By assumption, $h_2(T) = 0$, so $d_T(w) \geq 3$.
If $d_T(w) > 3$, consider $T'$ obtained by removing $u$.
Then $h_2(T') = h_2(T) = 0$ and $h_1(T') = h_1(T) - 1$ is divisible by $3$, so there exists an edge-separating family $\mathcal{F}'$ for $T'$ of size at most $2 h_1(T')/3 = \lceil 2 h_1(T)/3 \rceil - 1$.
We obtain an edge-separating family for $T$ from $\mathcal{F}'$ by adding the path $\{ uw \}$.
Thus $\spathC(T,E) \leq | \mathcal{F'}| + 1 \leq \lceil 2 h_1(T)/3 \rceil$, as desired.

Thus we can assume that $d_T(w) = 3$.
Let $u_1, u_2$ be the other two neighbors of $w$.
Consider $T'$ obtained by removing $u$ and $w$, and adding the edge $u_1 u_2$.
Thus $h_2(T') = h_2(T) = 0$ and $h_1(T') = h_1(T) - 1$ is divisible by $3$, so there exists an edge-separating family $\mathcal{F}'$ for $T'$ of size at most $2 h_1(T')/3 = \lceil 2 h_1(T)/3 \rceil - 1$.
We obtain an edge-separating family for $T$ from $\mathcal{F}'$ by replacing, in every path of $\mathcal{F}'$, the appearance of the edge $u_1 u_2$ by the path $u_1 w u_2$; and then adding to $\mathcal{F'}$ the path $u w u_1$.
It is straightforward to check this is an edge-separating family, and therefore $\spathC(T,E) \leq | \mathcal{F'}| + 1 \leq \lceil 2 h_1(T)/3 \rceil$, as desired.
\end{proof}

Now we consider trees with $h_1(T) \geq h_2(T) > 0$ and at least one useful leaf.
The idea is to construct a tree $T'$ with fewer leaves, to find an edge-separating system in $T'$, and to transform it into an edge-separating system of $T$.
The next definition captures this reduction, and the following lemma formalizes the transformation of one system to another.

\begin{definition}
    Let $T$ be a tree, $u$ a useful leaf of $T$, $w$ its unique neighbor, and $v$ a vertex of degree $2$ in $T$.
    We say that the pair $(u,v)$ is a \emph{reduction pair for $T$} if one of the following holds:
    \begin{enumerate}
        \item $w$ has degree at least $4$, or
        \item $w$ has degree $3$ and $v$ is not a neighbor of $w$.
    \end{enumerate}
    If $T$ has a reduction pair, we say $T$ is \emph{reducible}.
\end{definition}

\begin{lemma} \label{lemma:reducingreducibletrees}
    Let $T$ be a reducible tree.
    Then, there exists $T'$ with $h_1(T') = h_1(T)-1$ and $h_2(T') = h_2(T') - 1$ such that $\spathC(T,E) \leq \spathC(T',E) + 1$.
\end{lemma}

\begin{proof}
    Let $(u,v)$ be a reduction pair for $T$, let $w$ be the unique neighbor of $u$, and let $v_1, v_2$ be the two neighbors of $v$.
    We consider two cases depending on the two possible outcomes of the definition of reduction pair.

    Suppose first that $w$ has a degree at least $4$.
    Then let $T'$ be obtained by removing the leaf $u$, the vertex $v$, and then adding the edge $v_1 v_2$.
    Note that $h_1(T') = h_1(T)-1$ and $h_2(T') = h_2(T') - 1$.
    Let $\Fam'$ be an edge-separating-covering system of $T'$ of size $\spathC(T',E)$.
    We obtain a path system $\Fam$ in $T$ by extending each path of $\Fam'$ which contains the edge $v_1 v_2$ by replacing this appearance with $v_1 v v_2$; and finally adding the unique $u$-$v$-path $P$.
    Thus $|\Fam| = |\Fam'| + 1$.
    It is simple to check that $\Fam$ is an edge-separating-covering system, thus $\spathC(T,E) \leq |\Fam| = |\Fam'| + 1 = \spathC(T',E) + 1$, as desired.

    Now we suppose that $w$ has degree $3$ and $v$ is not a neighbor of $w$.
    Let $w_1, w_2$ be the two neighbors of $w$ which are not $u$.
    We let $T'$ be obtained by removing the vertices $u, w, v$, and then adding the edges $w_1 w_2$ and $v_1 v_2$.
    (This is well-defined because $\{w,v\} \cap  \{v_1, v_2\}=\emptyset$.)
    Note that $h_1(T') = h_1(T)-1$ and $h_2(T') = h_2(T') - 1$.
    Let $\Fam'$ be an edge-separating-covering system of $T'$ of size $\spathC(T',E)$.
    We obtain a path system $\Fam$ in $T$ by extending each path of $\Fam'$ which contains the edge $v_1 v_2$ by replacing this appearance with $v_1 v v_2$; by extending each path of $\Fam'$ which contains the edge $w_1 w_2$ by replacing this appearance with $w_1 w w_2$, and finally adding the unique $u$-$v$-path $P$.
    Thus $|\Fam| = |\Fam'| + 1$.
    Again, it is straightforward to check that $\Fam$ is an edge-separating-covering system, thus $\spathC(T,E) \leq |\Fam| = |\Fam'| + 1 = \spathC(T',E) + 1$, as desired.    
\end{proof}

If $T$ is a reducible tree and $T'$ is as in \Cref{lemma:reducingreducibletrees}, we say $T$ is \emph{reducible to $T'$}.
If we can apply reductions while bypassing the annoying tree $\Tannoying$, we could find an optimal path system by using induction on $h_2(T)$. The next fact is proven following a detailed analysis of the cases where $h_2(T) \leq 2$.
We show this in the three following lemmas, in which we carefully investigate cases where $1 \leq h_2(T) \leq 2$.

\begin{lemma} \label{lemma:h21-reducible}
    Let $T$ be a tree with at least one useful leaf and $h_1(T) \geq h_2(T) = 1$.
    Then one of the following is true: $T$ is reducible, $\spathC(T,E) = \lceil (2 h_1(T) + h_2(T))/3 \rceil$, or $T = \Tannoying$.
\end{lemma}

\begin{proof}
    Assume that $T$ is not reducible.
    Let $v$ be the unique vertex of degree $2$ in $T$, and let $y_1, y_2$ be its neighbors.
    If $y_1$ has degree $4$ or more, then there exists a leaf $u$ which together with $v$ forms a reduction pair, a contradiction.
    Then $y_1$ can have degree $3$ or $1$, by symmetry the same is true for $y_2$.
    If $y_1$ has degree $3$ and any of the neighbors of $y_1$ apart from $v$ is not a leaf, then there must exist a leaf $u'$ which is useful and is not a neighbor of $y_1$.
    Thus $(u', v)$ forms a reduction pair, a contradiction.
    Thus if $y_1$ has degree $3$, then both of its neighbors are leaves; by symmetry, the same is true for $y_2$.

    It cannot happen that both $y_1$ and $y_2$ are leaves, since then there would be no useful leaf.
    Thus, by symmetry, we are left with only two cases: $y_1$ has degree $3$ and $y_2$ is a leaf; or both $y_1$ and $y_2$ have degree $3$.

    Suppose first that $y_1$ has degree $3$ and $y_2$ is a leaf.
    In this case, the tree must have vertices $\{ u, u', y_1, v, y_2\}$ and edges $\{ u y_1, u' y_1, y_1 v, v y_2 \}$.
    Thus $h_1(T) = 3$ and $h_2(T) = 1$.
    Then $\{ u y_1 v, u' y_1 v y_2, v y_2 \}$ is an edge-separating-covering system of size $3 = \lceil (2 h_1(T) + h_2(T))/3 \rceil$, as required.

    Thus, we can suppose that both $y_1$ and $y_2$ have degree $3$.
    In this case, $T$ is isomorphic to $\Tannoying$, and we are done.
\end{proof}

\begin{lemma} \label{lemma:h22-notreducible}
    Let $T$ be a tree with at least one useful leaf and $h_1(T) \geq h_2(T) = 2$.
    If $T$ is not reducible, then $\spathC(T,V) = \lceil (2 h_1(T) + h_2(T))/3 \rceil$.
\end{lemma}

\begin{proof}
    Suppose $T$ is not reducible.
    Let $v_1$ and $v_2$ be the only vertices of degree $2$ in $T$.
    Let $P$ be the unique $v_1$-$v_2$-path in $T$.
    Let $x_1$ (resp. $x_2$) be the neighbor of $v_1$ (resp. $v_2$) which is not in $P$.
    If $x_1$ is not a leaf, then there exists a useful leaf $u'$ which is not a neighbor of $v_1$; and thus $(u', v_2)$ form a reduction pair, a contradiction.
    Thus $x_1$ is a leaf, and by symmetry $x_2$ is also a leaf.

    Note that $P$ cannot have length one, as in this case $T$ is equal to the path $x_1 v_1 v_2 x_2$ and there would be no useful leaf.
    Thus $P$ has a length of at least two, and therefore the neighbor $y$ of $v_1$ which is not $x_1$ satisfies $y \neq v_2$.
    
    Now suppose that $P$ has at least three edges.
    Since $y$ does not have degree $2$, there must be a neighbor $z$ of $y$ which is not in $P$.
    Then there must exist a leaf $u'$ in the tree generated by cutting the edge $y z$ which does not contain $v_2$; and $u'$ must be a useful leaf since its unique neighbor is not $v_1$ nor $v_2$.
    Also $u'$ has distance at least two to $v_2$, so $(u', v_2)$ form a useful pair, a contradiction.
    
    Thus $P$ has length exactly two, and indeed $P = v_1 y v_2$.
    Again, let $z$ be a neighbor of $y$ which is not in $P$.
    If $z$ is not a leaf, then we can find a useful leaf $u'$ as before and form a useful pair $(u', v_2)$.
    Therefore we have that $z$ is a leaf.
    This completely describes the tree, which must have vertices $\{ x_1, v_1, y, v_2, x_2, z\}$ and edges $\{ x_1 v_1, v_1 y, y v_2, v_2 x_2, yz \}$.
    Then $h_1(T) = 3$ and $h_2(T) = 2$.
    We have that $\{ x_1 v_1 y v_2, v_1 y v_2 x_2, v_1 y z \}$ form an edge-separating-covering system of $T$, and thus $\spathC(T,V) \leq 3 = \lceil (2 h_1(T) + h_2(T))/3 \rceil$, as required.
\end{proof}

\begin{lemma} \label{lemma:h22-reducible}
    Let $T$ be a tree with at least one useful leaf and $h_1(T) \geq h_2(T) = 2$.
    If $T$ is reducible, then one of the following is true: $T$ is reducible to a tree $T'$ which is not $\Tannoying$; or $\spathC(T,E) \leq \lceil (2 h_1(T) + h_2(T))/3 \rceil$.
\end{lemma}

\begin{proof}
    Suppose $T$ is reducible but only can be reduced to $\Tannoying$.
    In particular, this implies that $h_1(T) = h_1(\Tannoying) + 1 = 5$ and $h_2(T) = h_2(\Tannoying) + 1 = 2$.
    Thus $\lceil (2 h_1(T) + h_2(T))/3 \rceil = 4$.
    Hence, we can assume that $T$ cannot be reduced to a tree $T'$ which is not $\Tannoying$ and the task is to show that $\spathC(T,E) \leq 4$.

    Let $v_1, v_2$ be the two vertices of $T$ with degree $2$.
    We partition $V(T) \setminus \{ v_1, v_2\}$ into three sets:
    $A$ contains the vertices $x$ whose $x$-$v_2$-path in $T$ contains the vertex $v_1$;
    $C$ contains the vertices $x$ whose $x$-$v_1$-path in $T$ contains the vertex $v_2$;
    and $B$ contains all other vertices.
    It is possible that $B = \emptyset$, which happens if and only if $v_1$ and $v_2$ are neighbors.
    
    Let $u$ be a useful leaf of $T$ (by assumption, we know that there must be at least one).
    The unique neighbor $w$ of $u$ does not have degree $2$, and thus is neither $v_1$ nor $v_2$.
    It cannot be a leaf either, so $w$ must have degree at least $3$.
    Thus it must happen that $\{v, w\}$ is completely contained in $A$, or in $B$, or in $C$.
    We consider two cases.

    \medskip
    \noindent \emph{Case 1: $A \cup C$ does not contain a useful leaf.}
    Therefore, $B$ must contain a useful leaf, since $T$ has at least one useful leaf $v$.
    This assumption in particular implies that $B \neq \emptyset$, and therefore $v_1, v_2$ are not neighbors.
    Let $w$ be the unique neighbor of $v$, we know that $w \in B$ as well.

    Suppose first that $v_1, v_2$ are neighbors of $w$.
    If $w$ has degree $3$, then $B = \{u, w\}$.
    By assumption, $T$ is reducible, so it must contain a reduction pair.
    Since $A \cup C$ does not contain a useful leaf by assumption, the only possible reduction pairs are $(u, v_1)$ or $(u, v_2)$.
    But $w$ has degree $3$, $v_1, v_2$ are neighbors of $w$, and thus $(u, v_1)$, $(u, v_2)$ are not reduction pairs, a contradiction.
    This implies that $w$ has degree at least $4$.
    In this case, both $(u, v_1)$ and $(u, v_2)$ are reduction pairs.
    If we reduce $T$ using the pair $(u, v_2)$, we obtain a tree $T'$ which by assumption must be isomorphic to $\Tannoying$.
    Note that $d_{T'}(w) = d_T(w)-1 \geq 3$.
    Since $\Delta(\Tannoying) = 3$, it must happen that $3 = d_{T'}(w)$, and indeed $w$ must be adjacent to two leaves in $T'$ and to one degree-$2$ vertex.
    Note that in $T'$ we have that $w$ and $v_1$ are neighbors and have degree $3$ and $2$ respectively.
    This must imply that the component $C$ is actually a singleton vertex.
    By symmetry (by reducing using the pair $(u, v_1)$) we can deduce the same of $A$; but then this will imply that $h_1(T) \leq 3$, a contradiction.

    Therefore we can assume that one of $v_1$ or $v_2$ is not a neighbor of $w$.
    Without loss of generality we can assume $v_2$ is not a neighbor of $w$.
    Then $(u, v_2)$ is a reduction pair.
    We can reduce $T$ to a tree $T'$, which by assumption is isomorphic to $\Tannoying$.
    In $T'$ the vertex $v_1$ still has degree $2$, so it must be the only degree-$2$ vertex of $\Tannoying$.
    Thus $v_1$ is adjacent to two degree-$3$ vertices, each of which is adjacent to two leaves.
    In terms of $T$, this must imply that $C$ is a singleton vertex, and that one of the two neighbors $y_1, y_2$ of $v_2$ is actually $v_1$, but this contradicts the observation made at the beginning of Case 1.
    This finishes Case 1.
    
    \medskip
    \noindent \emph{Case 2: $A \cup C$ contains a useful leaf.}
    Without of loss of generality (by symmetry) we can assume that a useful leaf $v$ and its unique neighbor $w$ satisfy $\{ v, w \} \subseteq A$.
    In particular, we have that $v_2$ is not a neighbor of $w$, and thus $(u, v_2)$ is a reduction pair.
    Let $y_1, y_2$ be the two neighbors of $v_2$, where $y_1$ is closest to $v_1$.
    The analysis will change depending if $w$ has degree $3$ or at least $4$.
    
    \medskip
    \noindent \emph{Case 2.1: $w$ has degree $3$.}
    Let $w_1, w_2$ be the neighbors of $w$ which are not $u$.
    Without loss of generality we can assume that $w_2$ belongs to the unique $u$-$v_1$-path in $T$ and $w_1$ does not.
    After the reduction using $(u, v_2)$, the vertices $u$,$w$, $v_2$ are removed and the edges $w_1 w_2$ and $y_1 y_2$ are added to obtain a tree $T'$.
    By assumption, $T'$ is isomorphic to $\Tannoying$.
    Since $v_1$ still has degree $2$ in $T'$, it must be the only degree-$2$ vertex in $T'$ and it must be adjacent to two degree-$3$ vertices, each of which is adjacent to two leaves.

    Up to symmetry there are four possible options, depending if the edges $w_1 w_2$ and $y_1 y_2$ in $T'$ are adjacent to leaves, or to $v_1$, respectively.
    As we shall see, all cases will imply that the original tree $T$ contain another reduction pair which we can use to reduce it to another tree which is not $\Tannoying$.

    \begin{enumerate}
        \item Suppose first that $w_1 w_2$ is adjacent to an edge and $y_1 y_2$ is adjacent to an edge.
        This implies that (up to relabeling) $T$ has precisely the set of vertices $\{ u, w, w_1, w_2, v_1, v_2, y_1, y_2, u_2, u_3\}$ and its edges are \[ \{ uw, w_1w, w w_2, u_2 w_2, w_2 v_1, v_1 y_1, u_3 y_1, y_1 v_2, v_2 y_2 \}.\]
        In this case, $(u, v_1)$ forms a reduction pair where the reduced tree is not isomorphic to $\Tannoying$.

        \item Suppose that $w_1 w_2$ is adjacent to an edge and $y_1 y_2$ is adjacent to $v_1$.
        In this case, up to relabeling, we have that $y_1 = v_1$, so the vertices of $T$ are $\{ u, w_1, w, w_2, u_2, v_1, v_2, y_2, u_3, u_4 \}$ and the edges are
        \[ \{ uw, w_1w, w w_2, u_2 w_2, w_2 v_1, v_1 v_2, v_2 y_2, y_2 u_3, y_2 u_4 \}.\]
        In this case, $(u_3, v_1)$ forms a reduction pair where the reduced tree is not isomorphic to $\Tannoying$.

        \item Suppose that $w_1 w_2$ is adjacent to $v_1$ and $y_1 y_2$ is adjacent to $v_1$.
        In this case, up to relabeling, we have that $w_2 = y_1 = v_1$, so the vertices of $T$ are \[\{ u, w_1, w, w_2, v_1, v_2, y_2, u_2, u_3, u_4, u_5 \}\] and the edges are
        \[ \{ uw, w_1w, w v_1, v_1 v_2, v_2 y_2, u_3 w_1, u_2 w_1, u_4 y_2, u_5 y_2 \}.\]
        In this case, $(u_4, v_1)$ forms a reduction pair where the reduced tree is not isomorphic to $\Tannoying$.

        \item Suppose that $w_1 w_2$ is adjacent to $v_1$ and $y_1 y_2$ is adjacent to a leaf.
        In this case, up to relabeling, we have that $w_2 = y_1$, so the vertices of $T$ are precisely $\{ u, w, w_1, v_1, v_2, y_1, y_2, u_2, u_3, u_4 \}$ and the edges are
        \[ \{ uw, w_1w, w v_1, v_1 y_1, y_1 v_2, v_2 y_2, u_2 w_1, u_3 w_1, u_4 y_1 \}.\]
        In this case, $(u_2, v_1)$ forms a reduction pair where the reduced tree is not isomorphic to $\Tannoying$.
    \end{enumerate}

    Since this covers all cases, we have reached a contradiction, thus $w$ cannot have degree $3$.

    \medskip
    \noindent \emph{Case 2.2: $w$ has degree at least $4$.}
    After the reduction using $(u, v_2)$, the vertices $u$ and $v_2$ is removed, and the edge $y_1 y_2$ is added.
    By assumption, the obtained tree $T'$ is isomorphic to $\Tannoying$.
    Since $v_1$ still has degree $2$ in $T'$, it must be the only degree-$2$ vertex in $T'$ and it must be adjacent to two degree-$3$ vertices, each of which is adjacent to two leaves.
    Since $w$ has degree at least $3$ in $T'$, it must be adjacent to $v_1$ and have degree exactly $3$ in $T'$ and is adjacent to two leaves and $v_1$.
    Thus, in $T$, we have that $w$ has degree $4$ and is adjacent to $v_1$, to $u$ and two other leaves, say $u_2, u_3$.
    Thus in fact $A = \{u, u_2, u_3, w\}$ and $w$ is a neighbor of $v_1$.
    Here there are two cases, depending if the edge $y_1 y_2$ in $T'$ is adjacent to $v_1$ or not.

    Suppose that $y_1 y_2$ is adjacent to $v_1$ in $T'$.
    This implies that $y_1 = v_1$.
    In $T'$, $y_2$ is a neighbor of $v_1$, so it must have degree $3$ and be adjacent to two other leaves, say $u_3, u_4$.
    This implies that $T$ has vertex set \[\{ u, u_1, u_2, w, v_1, v_2, y_2, u_3, u_4\}\] and edges
    \[ \{ uw, u_1w, u_2w, w v_1, v_1 v_2, v_2 y_2, y_2 u_3, y_2 u_4 \}. \]
    But then $(u_3, v_1)$ is a reduction pair which reduces $T$ to a tree which is not $\Tannoying$, a contradiction.

    Only one case remains, where $y_1 y_2$ is not adjacent to $v_1$ in $T'$.
    Therefore $y_1 y_2$ is adjacent to a leaf in $T'$.
    This implies that, in $T$, $y_2$ is a leaf and there is another leaf $u_3$ which is adjacent to $y_1$.
    Therefore, $T$ has vertex set $\{ u, u_1, u_2, w, v_1, y_1, v_2, y_2, u_3\}$ and edges
    \[ \{ uw, u_1w, u_2w, w v_1, v_1 y_1, y_1 v_2, v_2 y_2, y_1 u_3 \}. \]
    Here is the only case where all reduction pairs actually yield trees which are isomorphic to $\Tannoying$.
    But here $\{ uwv_1y_1v_2, u_1 w v_1 y_1 v_2 y_2, u_2 w v_1, u_1 v w_1 y_1 u_3 \}$ is an edge-separating system of size $4$, and thus $\spathC(T, E) \leq 4$, as required.
    This finishes all cases and finishes the proof.    
\end{proof}

\begin{lemma} \label{lemma:caseanalysis}
    Let $T \neq \Tannoying$ be a tree with at least one useful leaf, $h_1(T) \geq h_2(T)$ and $h_2(T) \leq 2$.
    Then $\spathC(T,E) \leq \lceil (2 h_1(T) + h_2(T))/3 \rceil$. 
\end{lemma}

\begin{proof}
    The case $h_2(T) = 0$ follows from \Cref{lema-arista-h1mayorh2}.
    If $h_2(T) = 1$, then by the assumption of $T \neq \Tannoying$, \Cref{lemma:h21-reducible} implies that $T$ is reducible or $\spathC(T,E) = \lceil (2h_2(T) + h_1(T))/3 \rceil$.
    In the last case, we are done, so we can assume $T$ is reducible.
    Thus there exists $T'$ such that $h_2(T') = h_2(T)-1 = 0$ and $h_1(T') = h_1(T)-1$, and
    $\spathC(T,E) \leq \spathC(T',E) + 1$.
    Since $h_2(\Tannoying) = 1$, we also know that $T' \neq \Tannoying$.
    Thus we get that $\spathC(T',E) \leq \lceil (2h_2(T') + h_1(T'))/3 \rceil = \lceil (2h_2(T) + h_1(T))/3 \rceil - 1$,
    so $\spathC(T,E) \leq \lceil (2h_2(T) + h_1(T))/3 \rceil$, as desired.

    If $h_2(T) = 2$ and $T$ is not reducible, then we are done by \Cref{lemma:h22-notreducible}.
    Thus, we can assume $T$ is reducible.
    By \Cref{lemma:h22-reducible}, we know there exists a tree $T'$ with $T' \neq \Tannoying$ such that 
    $h_2(T') = h_2(T)-1 = 1$ and $h_1(T') = h_1(T)-1$, and
    $\spathC(T,E) \leq \spathC(T',E) + 1$.
    By induction, we know that $\spathC(T',E) \leq \lceil (2h_2(T') + h_1(T'))/3 \rceil = \lceil (2h_2(T) + h_1(T))/3 \rceil - 1$, so again we get $\spathC(T,E) \leq \lceil (2h_2(T) + h_1(T))/3 \rceil$, as desired.
\end{proof}

The next short lemma shows that when there are at least $3$ vertices of degree $2$, things become much easier.

\begin{lemma} \label{lemma:h23-reducible}
    Let $T$ be a tree with at least one useful leaf and $h_1(T) \geq h_2(T) \geq 3$.
    Then $T$ is reducible to a tree $T'$ with $T' \neq \Tannoying$.
\end{lemma}

\begin{proof}
    Suppose $T$ is not reducible.
    Let $u$ be a useful leaf of $T$, and let $w$ be its unique neighbor.
    By definition, $w$ does not have degree $2$, and it cannot be a leaf since then $h_2(T) \geq 3$ cannot hold. Thus $w$ has degree at least $3$.
    If $w$ has degree at least $4$, then for any vertex $v$ of degree $2$ we have that $(u,v)$ is a reduction pair, a contradiction.
    Thus $w$ has degree exactly $3$.
    Since $h_2(T) \geq 3$, there must exist a vertex $v$ of degree $2$ which is not a neighbor of $w$.
    Thus $(u,v)$ is a reduction pair, a contradiction.
    Then $T$ is reducible, and since $h_2(T) \geq 3$ and $h_2(\Tannoying) = 1$, it is not reducible to $\Tannoying$.
\end{proof}

We can now put all the ingredients together to finish the case where $h_1(T) \geq h_2(T)$.

\begin{lemma} \label{lemma:h1atleasth2}
    Let $T$ be a tree with $h_1(T) \geq h_2(T)$ and $T \neq \Tannoying$.
    Then \[\spathC(T,E) \leq \lceil (2h_1(T) + h_2(T)) / 3 \rceil.\]
\end{lemma}

\begin{proof}
    If all leaves are adjacent to a degree-$2$ vertex, then the result follows from \Cref{lema-arista-h1igualh2}.
    Thus we can assume in addition that $T$ has at least one useful leaf.

    We proceed by induction over $h_2(T)$.
    If $h_2(T) \leq 2$ then the result is true by \Cref{lemma:caseanalysis}.
    Thus, we can also assume that $h_2(T) = r \geq 3$, and by induction we can assume that the result holds for all trees $T'$ with $T' \neq \Tannoying$ and $h_1(T') \geq h_2(T')$ and $h_2(T') < r$.
    By \Cref{lemma:h23-reducible}, $T$ is reducible to a tree $T' \neq \Tannoying$.
    Thus $h_2(T') = h_2(T)-1 = r-1 < r$ and $h_1(T') = h_1(T)-1$, and
    $\spathC(T,E) \leq \spathC(T',E) + 1$.
    By the induction hypothesis, we get that $\spathC(T',E) \leq \lceil (2h_2(T') + h_1(T'))/3 \rceil = \lceil (2h_2(T) + h_1(T))/3 \rceil - 1$,
    thus we get $\spathC(T,E) \leq \lceil (2h_2(T) + h_1(T))/3 \rceil$, as desired.
\end{proof}

\subsection{More vertices of degree 2 than leaves}

We next consider the case when there are more vertices of degree $2$ than leaves.

\begin{lemma} \label{lema-arista-h1menorh2}
    Let $T$ be a tree such that $h_1(T)\leq h_2(T)$, then 
 \[ \spathC (T, E) \leq \left\lceil(h_1(T)+h_2(T))/2\right\rceil.\]
\end{lemma}
  
\begin{proof}
    We first consider the case where $h_1(T) + h_2(T)$ is even.
    For this, we proceed by induction in $h_2(T)-h_1(T)\geq 0$.
    In the base case, we have $h_2(T)-h_1(T) = 0$.
    This implies that $h_2(T) = h_1(T)$.
    In particular $T \neq \Tannoying$, and it also holds that $\left\lceil(h_1(T)+h_2(T))/2\right\rceil = \left\lceil(2h_1(T)+h_2(T))/3\right\rceil$.
    Then, in this case the lemma readily follows from Lemma~\ref{lemma:h1atleasth2}.
  
    For the inductive step, we assume that $h_2(T) - h_1(T) = 2r > 0$, and that the statement holds for each $T'$ as in the statement with $h_2(T') - h_1(T')$ even and less than $2r$.
    Since $h_2(T) - h_1(T) > 0$, we must have that $|E(T)| \geq 1$.
    Every tree with $|E(T)|\geq 1$ has at least two leaves.
    Since $h_2(T)>h_1(T)\geq 2$, we deduce that $h_2(T)\geq 3$.
    Therefore, there must exist two vertices $u,v \in V(T)$ of degree $2$ each, which are not adjacent to each other.
    
    Suppose $u_1, u_2$ are the neighbors of $u$ and $v_1, v_2$ are the neighbors of $v$, by assumption we have $v \notin \{u_1, u_2\}$ and $u \notin \{v_1, v_2\}$.
    We obtain $T'$ from $T$ by removing both $u$ and $v$ and adding the edges $u_1 u_2$ and $v_1 v_2$.
    Thus $h_1(T') = h_1(T)$ and $h_2(T') = h_2(T)-2$, and hence $h_2(T') - h_1(T') = 2r-2 < 2r$.
    By the inductive hypothesis, there exists an $E(T')$-separating-covering system $\mathcal{F}'$ of size $|\mathcal{F}'| \leq (h_1(T)+h_2(T)-2)/2 = \lceil (h_1(T)+h_2(T))/2 \rceil - 1$.
    
    In $T$, we obtain a path family $\Fam$ from $\Fam'$ as follows: for every path in $\Fam'$ which contains the edge $u_1 u_2$ we replace it with $u_1 u u_2$, and for every appearance of the edge $v_1 v_2$ we replace it with $v_1 v v_2$.
    To finish, we add to $\Fam$ the (unique) path $P$ in $T$ which goes from $u$ to $v$.
    We have that 
    $|\mathcal{F}|=|\mathcal{F}'|+1 \leq \frac{h_1(T)+h_2(T)}{2}=\left\lceil\frac{h_1(T)+h_2(T)}{2}\right\rceil$.
    
    It is easy to see that $\Fam$ covers $E(T)$, so it only remains to check that $\Fam$ separates $E(T)$.
    Indeed, let $e_1, e_2$ be two edges of $E(T)$.
    If both of them are in $E(T')$, then they are separated (since they were separated by $\Fam'$).
    Next, note that if $e$ is adjacent to $u$, then it must be contained in the paths of $\Fam$ which correspond to paths of $\Fam'$ which contain $u_1 u_2$; and similarly if $e$ is adjacent to $v$, then it must be contained in the paths of $\Fam$ which correspond to paths of $\Fam'$ which contain $v_1 v_2$.
    Therefore, in the case where $e_1$ is adjacent to $u$ and $e_2$ is not adjacent to $u$ we are done, again by the separating property of $\Fam'$.
    If both $e_1, e_2$ are adjacent to $u$, they are separated by $P$.
    An analogous reasoning works if $u$ is replaced by $v$, and this covers all cases.
    This finishes the induction.

    It thus only remains to consider the case where $h_2(T)+h_1(T)$ is odd.
    As we shall see, we can reduce this to the previous case.
    Take an arbitrary edge $uv \in E(T)$ and obtain a tree $T'$ by replacing it with a path $uwv$.
    This step creates a new vertex of degree $2$, hence $h_2(T') = h_2(T)+1$ and $h_1(T') = h_1(T)$.
    In particular, $h_2(T') \geq h_1(T')$ and $h_2(T')+h_1(T')$ is even.
    Thus, by using what we have already shown, we have that $T'$ has a path system $\Fam'$ which separates and covers $E(T')$ and
    \[ |\Fam'| \leq \left\lceil \frac{h_1(T')+h_2(T')}{2} \right\rceil = \left\lceil \frac{h_1(T)+h_2(T)+1}{2} \right\rceil = \left\lceil \frac{h_1(T)+h_2(T)}{2} \right\rceil, \]
    where in the last equality we used that $h_2(T)+h_1(T)$ is odd.
    We obtain a path system $\Fam$ in $T$ by replacing every appearance of the edges $uw$ and $wv$ by $uv$.
    This path system covers and separates $E(T)$, thus we are done.
\end{proof}

\subsection{Proof of Theorem~\ref{theorem:edgesepcover-trees}}

Now we have the ingredients to give the proof of \Cref{theorem:edgesepcover-trees}.

\begin{proof}[Proof of \Cref{theorem:edgesepcover-trees}]
Let $T$ be a tree which is not $\Tannoying$, $h_1 = h_1(T)$ and $h_2 = h_2(T)$.
\Cref{lem:cotainf} yields $\max \{ (2 h_1 + h_2)/3, (h_1+h_2)/2 \} \leq\spathC(T,E)$, it remains to show the other inequality.
If $h_1 \geq h_2$, then \Cref{lemma:h1atleasth2} implies the desired inequality.

Thus, we can suppose that $h_1<h_2$.
In this case, \Cref{lema-arista-h1menorh2} yields that
$\spathC(T,E)=\left\lceil(h_1+h_2)/2\right\rceil\leq \max \{ \lceil (2 h_1 + h_2)/3 \rceil, \lceil  (h_1+h_2)/2 \rceil \}$.
This covers all cases and finishes the proof.
\end{proof}

\section{Vertex-separating path systems in trees}\label{sec:Tvertices}

This section presents our study on $\spathC(T,V)$ and $\spathC(T, V\cup E^*)$. 
\subsection{Lower bounds}
We say that a path $P$ \emph{kisses} an edge $xy$ if $|V(P)\cap\{x,y\}|=1$. More over, we say that a vertex $v \in V(P)$ \emph{kisses} $xy$ if $V(P)\cap\{x,y\}=v$.
\begin{lemma}\label{every edge need to be kissed}
    Let $T=(V,E)$ be a tree and let $\mathcal{F}$ be a $V$-separating system.
    Then, for every edge $e$, there is a path in $\mathcal{F}$ kissing $e$.
\end{lemma}

\begin{proof}
Assume that there exists an edge $e=xy$ with no path in $\mathcal{F}$ kissing $e$.
Then, every $P\in \mathcal{F}$ satisfies $|V(P)\cap\{x,y\}|=0$ or $|V(P)\cap\{x,y\}|=2$.
In either case, $\mathcal{F}$ does not separate $x$ and $y$, which is a contradiction.
\end{proof}

\begin{lemma}
    Let $T=(V,E)$ be a tree and let $\mathcal{F}$ be a $V$-separating system.
    If $e=xy \in E$ is such that $d_T(x)=d_T(y)=2$, then $x$ or $y$ is the end of a path in $\mathcal{F}$.
\end{lemma}

\begin{proof}
    By Lemma \ref{every edge need to be kissed}, for every edge $e$, there is a path in $\mathcal{F}$ kissing $e$. Since $d_T(x)=d_T(y)=2$, either $x$ or $y$ is the end of the path in $\mathcal{F}$ kissing $e$.
\end{proof}

The following lemma gives a lower bound on $\spathC(T,V)$. This result is analogous to the lower bound for $\spathC(T,E)$ given in \Cref{lem:cotainf}.

\begin{lemma}\label{lem:cotainfsepararV}
    For any tree $T$,
    \[
        \spathC(T,V)\geq \max\left\{(h_1(T) + h_2^\ast(T))/2,(2h_1(T) + h_2^\ast(T))/3\right\}.\]
\end{lemma}

\begin{proof}
  Let $\mathcal{F}$ be a $V$-separating-covering system.
  We show first that we can assume that $\mathcal{F}$ does not have a path $P = \{v\}$ of length zero, where $v$ is a leaf or a vertex of degree $2$ in $T$.
  Indeed, assume otherwise.
  If $v$ is a leaf, then define $T'$ by removing $v$ from $T$; and if $v$ is a vertex of degree two with neighbors $v_1, v_2$, then define $T'$ by removing $v$ and adding the edge $v_1 v_2$.
  In any case, we have that $h_1(T') + h_2^\ast(T') \geq h_1(T)+h^\ast_2(T)-1$ and $2h_1(T') + h^\ast_2(T') \geq 2h_1(T)+h^\ast_2(T)-2$.
  Given any $V(T')$-separating-covering system $\Fam'$ of $T'$ we can naturally obtain a system $\Fam$ of $T$ by replacing, if necessary, the appearance of the edge $v_1 v_2$ in any path by $v_1 v v_2$; and then adding the path $P = \{v\}$ to $\Fam$.
  Thus $|\Fam| = |\Fam'|+1$, so if $T'$ satisfies the lower bound then the same inequality will hold for $\Fam$.
  
To prove the first inequality, we show that $\mathcal{F}$ produces at least $h_1(T) + h_2^\ast(T)$ ends of paths in $T$, \textit{i.e.,}~vertices that are the extreme of at least one path.  
Let $P_i\in \mathcal{I}$, so $P_i$ is a maximum bare path in $T$ such that $|E(P_i)|\geq 2$.
We claim that $\mathcal{F}$ produces at least $|E(P_i)|-2$ ends of paths in $P_i$.
Indeed, every end of a path in $P_i$ kisses at most one interior edge of $P_i$ (we use here that no path of $\Fam$ has length zero), and there are exactly $|E(P_i)|-2$ interior edges of $P_i$.
By the pigeonhole principle, if there are less than $|E(P_i)|-2$ ends of paths in $P_i$, then there is at least one interior edge that is not kissed by any path, and this would contradict Lemma~\ref{lem:cotainfsepararV}. 

Now, consider $P_i \notin \mathcal{I}$, so either $P_i$ does not contain a leaf and has exactly one edge, or it contains a leaf.
We want to show that $\Fam$ produced at least $|E(P_i)| - 1$ ends of paths in $P_i$.
If $|E(P_i)| = 1$ there is nothing to show, so we can assume that $P_i$ contains a leaf.
Say $u$ is the leaf and $v$ is its unique neighbor.
The previous argument still applies, and shows that to separate the interior vertices of $P_i$, $\Fam$ produces at least $|E(P_i)| - 2$ ends of paths in $P_i$.
Also, the edge $uv$ needs to be kissed by some path in $\Fam$.
This path cannot be $\{u\}$ (because we ruled out the paths of length zero) and cannot contain $uv$, so it has an end of path precisely in the vertex $v$: this end of a path only contributes to the kissing of the edge $uv$ and was not counted before.
Including this end of a path in the count, we obtain that $\Fam$ produces at least $|E(P)|-1$ end of paths to separate all vertices in $P_i$.
Finally, every leaf needs to be the end of a path in $\mathcal{F}$ since they need to be covered. 
Therefore, $\mathcal{F}$ produces at least $h_1(T)+h_2^\ast(T)$ ends of paths and, in consequence, $\Fam$ has at least $\left\lceil(h_1(T) + h_2^\ast(T))/2\right\rceil$ paths.

The analysis to obtain the second inequality is analogous to the analysis presented in the proof of Theorem 6.1 in \cite{FKKLN2014}. The crucial point in understanding this lower bound is to note that if $u$ and $v$ are two leaves in $T$, and $P\in\Fam$ is a path whose ends are $u$ and $v$, another path different from $P$ must end in either $u$ or $v$ as $P$ does not separate them. 

By contradiction, assume the inequality is false. Since 
$|\Fam|$ is an integer, we have 
$|\Fam|< \frac{2h_1(T) + h_2^\ast(T)}{3}$.
On the other hand,  $\Fam$ produces at most $2|\Fam|$ ends of paths, and as we see earlier, it produces at least $h_1(T)+h_2^\ast(T)$ ends of paths. Therefore, we have
$h_1(T) + h_2^\ast(T) < 2\cdot (2h_1(T) + h_2^\ast(T))/3,$
thus $h_2^\ast(T) < h_1(T)$. Let us say $ h_1(T)=h_2^\ast(T) + l$. Hence, $\Fam$ produces at least $h_1(T)+h_2^\ast(T) = 2h_2^\ast(T) +l$ ends of paths in $T$. 

If we pair these $2h_2^\ast(T) +l$ ends of paths once at a time covering first the $h_2^\ast(T)$ vertices of degree two that are end of a path, we will always end up with at least $l$ leaves that we will have to pair just between them.
However, as we said before, leaves need special treatment.
If $u$ and $v$ are two leaves in $T$, and $P$ is the path in $\Fam$ whose ends are $u$ and $v$, then a different path in $\Fam$ needs to end in either $u$ or $v$ as they need to be separated.
The best way to do this is as follows.
For every three leaves, $u,v$ and $w$, include the $u$-$v$-path and the $v$-$w$-path.
Therefore, $\Fam$ contains at least two paths for every three leaves.
Thus, $\Fam$ contains at least the first $h_2^\ast(T)$ paths to cover the $2h_2^\ast(T)$ ends of paths plus the at least $2l/3$ paths to cover the remaining $l$ leaves. Then, 
$
|\Fam| \geq h_2^\ast(T) +\frac{2l}{3}=\frac{3h_2^\ast(T) + 2l }{3} = \frac{2h_1(T)+h_2^\ast(T)}{3}>|\Fam|,
$
which is a contradiction.  
\end{proof}

\subsection{Sharp instances}
Next, we examine some instances of trees where we can determine $\spathC(T, V)$ or $\spathC(T, V \cup E^\ast)$ precisely.

The following corollary is a direct consequence of Lemma \ref{lem:cotainfsepararV} and the Bunch construction~\ref{caso-planar-2}. 
Let $T=(V, E)$ be a tree with $h_2=0$ and at least two bunches of leaves, all having size at least three. If $T$ has no vertex with degree two,  then $h_2^\ast=0$.
Therefore, according to Lemma \ref{lem:cotainfsepararV}, any $V$-separating-covering system has size at least $\left\lceil\frac{2h_1}{3}\right\rceil$.
Since all the bunches of leaves of $T$ have size at least three, the Bunch construction (\Cref{caso-planar-2}) provides a $V$-separating-covering system of size $\left\lceil\frac{2h_1}{3}\right\rceil$.

\begin{corollary}
    Let $T=(V, E)$ be a tree with $h_2=0$ and at least two bunches of leaves, all having size at least three.
    Then,
        $\spathC(T, V)   = \left\lceil\frac{2h_1}{3}\right\rceil$. \hfill  
\end{corollary}

\begin{lemma}\label{lem:svg2}
Let $T=(V, E)$ be a tree on at least $4$ vertices with only vertices of degree one or three. Then,  $\spathC(T,V\cup E^*) \geq h_1(T) = |E^*(T)| +3.$
\end{lemma}

\begin{proof} 
Let $\mathcal{C}_{1,3}$ be the family of trees with degree one and degree three vertices only and at least $4$ vertices. Consider the following inductive scheme:
 The 3-leaf star $K_{1,3}$ belongs to $\mathcal{C}_{1,3}$; and    given a tree  $T \in \mathcal{C}_{1,3}$, and a leaf $r$ in $T$, add two new leaves to $r$ to obtain a new tree in $\mathcal{C}_{1,3}$.
A tree can be generated in this fashion if and only if it belongs to $\mathcal{C}_{1,3}$.
Using this, an easy induction (with $K_{1,3}$ as base) shows that $h_1(T) = |E^*(T)|+3$ for any $T$ in $\mathcal{C}_{1,3}$.

We also show the inequality by induction starting from $K_{1,3}$. An exhaustive search shows that $\spathC(K_{1,3},V\cup E^*) \geq h_1(K_{1,3})=3$. 

Let $T$ be a tree in $\mathcal{C}_{1,3}$, assume that $\spathC(T,V\cup E^*)\geq h_1(T)$ and let $T'$ be a tree obtained from $T$, connecting two new leaves $u$ and $v$ to a leaf $r$ of $T$.
Suppose that $\spathC(T',V\cup E^*)< h_1(T')=h_1(T)+1$. Hence, there is a family $\mathcal{F}'$ that covers and separates $V(T')\cup E^*(T')$ such that $|\mathcal{F}'|\leq h_1(T)$.
In $T'$, $r$ has three neighbors, $u, v$, and $p$. The edge $pr$ is an interior edge in $T'$. Therefore, $\mathcal{F}'$ contains a path containing $r$ but not the edge $pr$, say $P^*$. Now, $\mathcal{F}'\setminus P^*$ separates $V(T)\cup E^*(T)$ since $pr$ is not an interior edge in $T$ and therefore $r$ does not need to be separated from $pr$. But this is a contradiction because, in that case, we obtain $\spathC(T,V\cup E^*)\leq|\mathcal{F}'\setminus P^*|\leq h_1(T)-1< h_1(T)$.
\end{proof}

What follows is a direct consequence from Lemma \ref{lem:svg2} and the Planar construction (\Cref{Construction-planar}).
\begin{corollary}
   If $T$ is a tree on at least $4$ vertices with only vertices of degree one or three,
 then
   $\spathC(T,V\cup E^*) = h_1(T)$.  
\end{corollary}

Let $\mathcal{C}_{1,3}^\ast$ be the following family of trees. Starting from a tree
$T=(V, E)$ in $\mathcal{C}_{1,3}$, we construct a tree $T^\ast$ in
$\mathcal{C}_{1,3}^\ast$ by replacing each interior edge $uv\in E^*(T)$ by a path of
length two $ux_{uv}v$, where $x_{uv}$ is a new vertex that is adjacent only to $u$ and
$v$ in $T^\ast$.
A $(V\cup E^*)$-separating-covering family $\Fam$ for $T$ can be modified to obtain a
$V$-separating-covering family $\Fam^\ast$ for $T^\ast$ and vice versa, such that $|\Fam|
= |\Fam^\ast|$.
Therefore, we have the following corollary.   

\begin{corollary} \label{corollary:c13star}
If $T^\ast$ is in $\mathcal{C}_{1,3}^\ast$, then $
    \spathC(T^\ast,V) = h_1(T^\ast)$. 
    \hfill 
\end{corollary}

\subsection{Proof of Theorem~\ref{theorem:edgevertexsep-trees}}
In this subsection, we will prove \Cref{theorem:edgevertexsep-trees}.
The lower bound is shown in Lemma \ref{lem:cotainfsepararV}. Hence, we need to show the upper bound. 

\begin{proof}[Proof of Theorem~\ref{theorem:edgevertexsep-trees}]

First, we consider the case when $T$ is the \emph{star} with $n\ge 4$ leaves. In this case, the Bunch construction has the correct size and separates $V$.

In order to prove the upper bound when the tree has at least one interior edge, we define a new construction. Let $T=(V,E)$ be a tree and $\mathcal{P} = \{P_1,\ldots, P_r\}$ be the family of maximal bare paths in $T$ that partitions $E$. We construct a $V$-separating-covering system as follows. \medskip

\noindent \emph{Step 1:}  We transform $T$ into a tree $T'$ by contracting every path $P_i$ to an edge $e_i$. Note that $T'$ is a tree with no vertices of degree two and $h_1(T')=h_1(T)$.
Therefore, we can apply the Bunch construction (\Cref{caso-planar-2}) to $T'$ to obtain a system $\Fam'$ of size $\lceil 2h_1(T')/3\rceil = \lceil2h_1(T)/3\rceil $. Lemma~\ref{lema-planar-racimos} implies that $\Fam'$ is a $V(T')\cup E^\ast(T')$-separating-covering system. \medskip

\noindent \emph{Step 2:}
We transform $\Fam'$ into a $V(T)$-separating-covering system. Each vertex in $T'$ is also a vertex in $T$. Therefore, if we replace each $u$-$v$-path in $\mathcal{F}'$ with the corresponding $u$-$v$-path in $T$, we obtain a system $\mathcal{F}^*$ that covers $V(T)\cup E(T)$ and separates all vertices with degree different than two. Note that $|\mathcal{F}'|=|\mathcal{F}^*|$. 

Now, for each vertex $u\in P_i$ of degree two, we have that $\Fam^*(u)=\Fam'(e_i)$.
Since $\mathcal{F}'$ separates $E(T')$, $u$ and $v$ are separated in $\mathcal{F}^*$ when they belong to different bare paths. \medskip

 \noindent \emph{Step 3:} In order to separate degree-two vertices in $T$ that belong to the same bare path,  we need to add some new paths to $\mathcal{F}^*$. We proceed as follows. For each bare path $P_i\in \mathcal{I}$, we consider all its degree two vertices to be unmarked except for one (any such vertex will do).
 Now, we pick two unmarked vertices of degree two belonging to different bare paths, say $u$ and $v$, and we add the $u$-$v$-path in $T$ to $\mathcal{F}^*$, then we mark $u$ and $v$.
 We do this in such a way that unmarked vertices in the same bare path are always consecutive.
 We repeat this process until all degree two vertices in $T$ are marked, or until all unmarked vertices of degree two belong to the same bare path. \medskip 
 
\noindent \emph{Refinement of step 3:} The paths just added do not necessarily separate the marked degree two vertices inside a given bare path, and we need to perform a refinement. 

 Let $u$ and $v$ be two such vertices.
 Assume $u$ is not one of the initially marked vertices.
 Let $P$ be the $u$-$x$-path added when $u$ was marked, and let $Q$ be the $v$-$y$-path added when $v$ was marked (if any). 
 We have that $u$ and $v$ are not separated if and only if  $v\in P$ and $u\in Q$ (if $Q$ exists), that is, paths $P$ and $Q$ overlap inside the bare path.
 We delete $P$ and $Q$ from the family and add the $u$-$y$-path and the $v$-$x$-path in $T$, say $P'$ and $Q'$, respectively.
 If $Q$ and $y$ do not exist, only add $Q'$.
 Now, $P'$ and $Q'$ run in the opposite direction and do not overlap, thus $u\notin Q'$.
 We repeat this refinement until eliminating all the overlapping paths inside each bare path, and in such a way that the initially marked vertex of each $\mathcal{I}$ path is not contained in any of the paths that do not come from $\Fam'$. \medskip
 
 \noindent \emph{Step 4:} If there are still some unmarked degree-two vertices, we use the construction presented in \cite{FoucaudKovse2013} to separate them and add all the corresponding paths to $\mathcal{F}^*$.
 This will work because these vertices are all consecutive and belong to the same bare path.
 If we have $k$ of these last vertices, this construction adds $\left\lceil\frac{k+1}{2}\right\rceil$ new paths, and these paths only cover these $k$ vertices.

 Call $\mathcal{F}$ the family at the end of all steps.
We prove that $\mathcal{F}$ is a $V$-separating-covering system.
For simplicity, we say that $u$ \emph{belongs} to $e_i$ in $T'$, when $u \in P_i$ and $u$ has degree two. 

To see that $\Fam$ covers $V(T)$, note that if $u$ has a degree different than two in $T$, then $\emptyset\neq \Fam'(u)\subseteq \Fam(u)$\footnote{Even though $\Fam$ and $\Fam'$ are families of paths defined in different trees, we consider the obvious map between paths in $T'$ and paths in $T$ to say $\Fam'(u)\subseteq \Fam(u)$.}.
On the other hand, if $u$ is a degree two vertex in $T$ that belongs to $e_i$ in $T'$, then $\emptyset \neq \Fam'(e_i)\subseteq \Fam(u)$.
Therefore, every vertex in $T$ is covered by $\mathcal{F}$. 
     
  In order to see that $\Fam$ separates $V(T)$, assume by contradiction that there are two vertices $u$ and $v$ in $T$ such that $\Fam(u)=\Fam(v)$.
  First, we see that neither $u$ nor $v$ can have a degree different than two.
  Indeed, if $u$ and $v$ have degree different than two, then $\Fam'(u)\subseteq \Fam(u)$, and $\Fam'(v)\subseteq \Fam(v)$.
  Since $\Fam'$ separates $u$ and $v$, they are also separated in $\Fam$.
  On the other hand, if $v$ has degree two in $T$ and it belongs to $e_i$ in $T'$, then $\Fam'(e_i)\subseteq \Fam(v)$, and since $\Fam'$ separates $e_i$ from every vertex $u$, except if $u$ is an adjacent leaf, we are done because  in that case, steps 3 and 4 add the necessary paths to separate $v$ from $u$. 
  
Hence, $u$ and $v$ are two vertices of degree two in $T$.
If $u$ and $v$ belong to different bare paths $P_i$ and $P_j$, respectively, we know there is a path $P$ in $\Fam'(e_i)\Delta \Fam'(e_j)$, because $\Fam'$ separates $E(T')$.
Thus $\Fam$ separates $u$ and $v$ as well.
Therefore, $u$ and $v$ belong to the same bare path $P_i$.
But in that case, the new paths added in Step 3 and Step 4 separate them.
Therefore, these two vertices cannot exist and $\Fam$ is a $V(T)$-separating-covering system.
  
To conclude the proof, note that \[\spathC(T,V) \leq |\Fam|=\lceil 2h_1(T)/3\rceil + \lceil \left(h_2^\ast(T)+1\right)/2\rceil\] which are the paths in $\Fam'$ (first term) plus the paths added in Step 3, the refinement of Step 3, and Step 4 (second term of the sum).
\end{proof}

\section{Separating vertices in random graphs}\label{sec:randomgraphs}

This section presents our study of $\spathC(G, V)$ when $G = G(n,p)$ is a binomial random graph. 
We recall that $\spathC(K_n, V) = \lceil \log_2 (n+1) \rceil$, as shown in \cite{FoucaudKovse2013}, is done by finding a separating set family of $V(K_n)$ and then a spanning path in each set.
We use a similar strategy.

\begin{proof}[Proof of Theorem~\ref{theorem:random}]
    We prove item \ref{item:randomupper} first.
    Suppose $n = 2k+q$ with $q \in \{0,1\}$.
    Let $A, B, C \subseteq V(G)$ be a partition of $V(G)$ into sets of size $k, k, q$ respectively.
    Let $t = \lceil \log_2 k \rceil$ and let $\mathcal{Q} = \{ Q_1, \dotsc, Q_t \}$ be a separating set system of the set $[k]$
    (i.e. $Q_1, \dotsc, Q_t \subseteq [k]$, and for each distinct $i, j \in [k]$ there exists $\ell \in [t]$ such that $|Q_\ell \cap \{i, j\}| = 1$).
    Enumerate $A = \{a_1, \dotsc, a_k\}$, $B = \{b_1, \dotsc, b_k\}$ arbitrarily.
    For each $i \in [t]$, let $S_i \subseteq V(G)$ be
    $ S_i = \{ a_j : j \in [k] \cap Q_i\} \cup \{ b_j : j \in [k] \setminus Q_i\}$, 
    and set $S_{t+1} = A$.
    If $C$ is non-empty, define $S_{t+2} = C$.
    
    We claim that $\{S_1, \dotsc, S_{t+2}\}$ is a separating set system of $V(G)$ of size at most $\lceil \log_2 n \rceil + 1$.
    Note first that $k \leq n/2$, so $t+2 = \lceil \log_2 k \rceil + 2 \leq \lceil \log_2 n \rceil+1$.
    Now we show that $\{S_1, \dotsc, S_{t+2}\}$ separates the vertices of $V(G)$.
    Indeed, let $u, v$ be arbitrary and distinct vertices of $V(G)$.
    If $\{u, v\} \cap C \neq \emptyset$, then $\{u,v\}$ are separated by $S_{t+2} = C$.
    If $u \in A, v \in B$, then $\{u,v\}$ are separated by $S_{t+1} = A$.
    If $u, v \in A$, then there exist distinct $i, j$ such that $u = a_i, v = a_j$.
    Since $\mathcal{Q}$ is a separating set system of $[k]$, there exists $Q_\ell$ such that $|\{ i, j\} \cap Q_\ell| = 1$, but by construction this implies that $|\{a_i, a_j\} \cap S_\ell| = 1$, so $u, v$ are separated by $S_\ell$.
    The case where $u, v \in B$ is analogous.
    
    To conclude, we show that w.h.p., for every $i \in [t+1]$, there exists a path $P_i \subseteq G$ with $V(P_i) = S_i$.
    Observe that $|S_i| = k$ for each $i \in [t+1]$.
    We argue that w.h.p., for each $i \in [t+1]$, the induced graph $G[S_i]$ is Hamiltonian.
    Indeed, fix $i \in [t+1]$ and let $G_i := G[S_i]$.
    $G_i$ corresponds to a binomial random graph on $k$ vertices with probability $p$.
    By a result of Alon and Krivelevich~\cite{AlonKrivelevich2020}, we know that
    \begin{eqnarray*}
        \prob[G_i \text{ is not Hamiltonian}]
         &=& (1 + o(1))\prob[ \delta(G_i) < 2 ] \\
         &=& \Theta( k p (1 - p)^k)\\ &\leq& O( n e^{-pk}),
    \end{eqnarray*}
    where in the last step we used $(1 - p)^k \leq e^{-pk}$ and $kp \leq k \leq n$.
    Since $p \geq (2 \ln n + \omega(\ln \ln n)) / n$ and $k \geq (n-1)/2$, we conclude
    $e^{-pk} = o((n \log n)^{-1})$, and thus $\prob[G_i \text{ is not Hamiltonian}] = o((\log n)^{-1})$.
    Thus we can use a union bound over the $t+1 = O(\log n)$ events to conclude that w.h.p. each $G[S_i]$ is Hamiltonian, as required.
    
    Now we prove item \ref{item:randomlower}.
    Note first that, deterministically, if $I$ is the set of isolated vertices of a graph $G$,
    then $\spathC(G,V) \geq |I|$ (since an isolated vertex can only be covered by a one-vertex path using that vertex).
    Thus it is enough to see that in this range of probability the random graph $G$ has $\omega( \ln n)$ isolated vertices.
    Let $X$ be the random variable counting the number of isolated vertices of $G$.
    Standard calculations (see, e.g.~\cite[Example 6.28]{JansonLuczakRucinski2000}) show that for $p = (\ln n - c)/n$, we have $\expectedvalue[X] \approx e^{c}$ and $\variance[X] = \expectedvalue[X] + o(1)$.
    In particular, if $\expectedvalue[X] = \omega(1)$ we have $\variance[X] = o(E[X]^2)$,
    and thus by Chebyshev's inequality w.h.p. we have $X = (1 + o(1)) \expectedvalue[X]$.
    By using $p = (\ln n - \omega(\ln \ln n))/n$ we have that $\expectedvalue[X] = \omega( \ln n)$,
    and by the previous argument we have $X = \Omega( \expectedvalue[X]) = \omega( \ln n)$, as desired.
\end{proof}

\section{Future work} \label{sec:conclusions}

The task of determining an exact value for $\spathC(T, V)$ and for $\spathC(T, V\cup E)$ for every tree $T$, as we did for $\spathC(T, E)$, remains open.
On the other hand, if $\spathC(n, S)$ denotes the maximum of $\spathC(G, S)$ taken over all $n$-vertex graphs $G$, then, an interesting question would be to find the values of $\spathC(n, V \cup E)$.
Note that if $\mathcal{F}$ is an $E$-separating system of $G$, we can add to $\Fam$ the family of all zero-length paths $\{ v : v \in V(G) \}$ to obtain a $(V \cup E)$-separating system of size $|\Fam| + n$.
Using this and the result of $\spathC(n, E) \leq 19n$ mentioned in the introduction, we  can deduce that $\spath(n, V \cup E) \leq 20n$ holds, and to narrow this further could be interesting.

Complexity questions also remain open. The most general question is to find the complexity of the following problem: given a graph $G=(V, E)$, a set $S\subseteq V\cup E$, and an integer $k$, is it true that $\spathC(G, S)\leq k$? All avenues indicate that this problem is NP-complete. We also believe that deciding if $\spathC(G, E)\leq k$ and deciding if $\spathC(G, V)\leq k$ should be NP-complete.

\sloppy\printbibliography

\end{document}